\documentclass[12pt]{amsart}
\usepackage{a4}
\usepackage{amssymb}
\usepackage{amsmath}
\usepackage{amsthm}
\usepackage{amstext}
\usepackage{amscd}
\usepackage{latexsym}
\usepackage{graphics}
\usepackage{color}

\swapnumbers
\theoremstyle{plain}
\newtheorem{thm}{Theorem}[section]

\newtheorem{lem}[thm]{Lemma}
\newtheorem{cor}[thm]{Corollary}

\theoremstyle{definition}

\newtheorem{point}[thm]{}

\newcommand{\inj}{\hookrightarrow}

\newcommand{\intersection}{\cap}

\newcommand{\Div}{{\rm Div}}

\newcommand{\Spec}{{\rm Spec \,}}

\newcommand{\Char}{{\rm char}}

\newcommand{\Gal}{{\rm Gal}}

\renewcommand{\tilde}{\widetilde}

\newcommand{\sE}{{\mathcal E}}
\newcommand{\sF}{{\mathcal F}}
\newcommand{\sG}{{\mathcal G}}

\newcommand{\sO}{{\mathcal O}}

\newcommand{\C}{{\mathbb C}}

\newcommand{\G}{{\mathbb G}}

\newcommand{\Z}{{\mathbb Z}}

\input{xy}
\xyoption{all}
\begin{document}
\title{The Brauer group of a smooth orbifold}
\author{Amit Hogadi}
\maketitle
\begin{abstract}
Let $k$ be a field and $X/k$ be a smooth quasiprojective orbifold. Let $X\to \underline{X}$ be its coarse moduli space. In this paper we study the Brauer group of $X$ and compare it with the Brauer group of the smooth locus of $\underline{X}$.
\end{abstract}

\section{Introduction}
\noindent By an {\bf orbifold} over a field $k$ we mean a separated DM stack over $k$ having trivial isotropy group at its generic points. We say (by abuse of language) that $X$ is quasiprojective, if its coarse moduli space is quasiprojective. Typical example of an orbifold is the quotient stack $[Y/\Gamma]$ where $Y/k$ is a quasiprojective variety and $\Gamma$ is a finite group acting faithfully on $Y$. In this case the coarse moduli space of $[Y/\Gamma]$ is nothing but the geometric quotient of $Y$ by $\Gamma$. However, there also exist interesting orbifolds which are not global quotients of varieties by finite groups. \\

\noindent In this paper we study the Brauer group of a smooth $k$-orbifold and compare it with the Brauer group of the smooth locus of its coarse moduli space. If $X/k$ is a smooth orbifold, then one can show that the Brauer group of $X$ injects into the Brauer group of the function field of $X$, which is the same as the function field of its coarse moduli space. Moreover, the Brauer group of the smooth locus of the coarse moduli space also injects into the Brauer group of the function field. The main goal of this paper is to compare these two groups. As a result of this comparison we will see that certain 'ramified' Brauer classes on the coarse moduli space may become 'unramified' when pulled back to the orbifold. \\

\noindent Brauer group of orbisurfaces have been studied before in \cite{lieblich} where they have been used for splitting ramifications of Brauer classes and to prove the period index conjecture for the function field of a surface over a finite field. \\

\noindent To state the precise result we first set the following notation.
\begin{point}[Notation]
The coarse moduli space of any DM stack $X$ will be denoted by 'underlining' the same symbol, i.e.$X\to \underline{X}$. For any point $p\in X$, (or equivalently a point in the $\underline{X}$) $s_p$ will denote the order of the isotropy group at $p$. For any stacks $Z$,  $|Z|$ will denote the underlying Zariski topolocial space of $Z$ and $Z^{(1)}$ will denote the set of codimension one points of $Z$. If $Z$ is an integral scheme or an integral orbifold, the function field of $Z$ will be denoted by $\kappa(Z)$. For any stack $X/k$, $Br(X)'$ will denote the subgroup of $Br(X)$ consisting of those elements whose order is coprime to $\Char(k)$. In particular, if $\Char(k)=0$, then $Br(X)'=Br(X)$.
\end{point}

\noindent The main theorem of this paper is the following. 
\begin{thm}\label{thm:main}
Let $k$ be a field and $X/k$ be a smooth quasiprojective orbifold. $h:X\to \underline{X}$ be its coarse moduli space and $\underline{X}_{sm}$ be the smooth locus of $\underline{X}$. Then
$$ Br(X)' = \{ \alpha \in Br(\kappa(X))' \ | \ s_p\cdot \alpha \ \text{is unramified at } p , \ \forall \ p \in \underline{X}_{sm}^{(1)}\}$$
\end{thm}

\noindent The following is a special case of the above theorem: Let $Y/\C$ be a smooth quasiprojective variety and let $\Gamma$ be a finite group acting on $Y$ in such a way that the action is free on an open subset $U$ of $Y$ whose complement has codimension at least two. Then the $\Gamma$-equivarient Brauer group of $Y$ is naturally isomorphic to the Brauer group of the smooth locus of the geometric quotient of $Y$ by $\Gamma$.\\

\noindent We have the following easy corollary of Theorem (\ref{thm:main}). 

\begin{cor}
Let $X/k$ be a smooth quasiprojective variety and $\alpha \in Br(\kappa(X))'$. Assume that the ramification locus of $\alpha$ is a simple normal crossing divisor. Then there exists a smooth orbifold $\tilde{X}$ having coarse moduli space $X$ such that the pull back of $\alpha$ is unramified on $\tilde{X}$.
\end{cor}
\begin{proof}
This follows immediately from Theorem (\ref{thm:main}) and (\cite{matsuki-olsson},$4.1$).
\end{proof}

\noindent One of the mysterious properties of the Brauer group of a regular scheme $X$ is that $U\to Br(U)$ is a sheaf in the Zariski topology. This is a consequence of the following vanishing theorems
\begin{enumerate}
 \item $R^1\check{q}_*\G_m=0$, where $\check{q}$ is the continuous map from the etale site of $X$ to the Zariski site of $X$.
 \item $H^i(X_{Zar},\G_m)=0$ for $i\geq 1$.
\end{enumerate}
For the proof of Theorem (\ref{thm:main}), it will be important to prove (see Theorem (\ref{thm:brauersheaf})) that even when $X$ is a smooth orbifold, $U\to Br(U)$ is a Zariski sheaf. However in this case we note that the vanishing of $R^1\check{q}_*\G_m$ is no longer true. However we will show that the argument in the scheme case still survives because of a slightly weaker vanishing (see \ref{lem:vanishing}) and that $H^i(X_{Zar},\G_m)=0 \ \forall \ i>0$ holds after possibly throwing out a codimension two subset of $X$, which is ok because of purity (see \ref{thm:purity}). \\

\noindent The main idea of the proof is to reduce (\ref{thm:main}), using purity and Theorem (\ref{thm:brauersheaf}), to a problem of Galois cohomology of complete discrete valued fields. \\

\noindent In the next section we discuss some basic facts about Brauer group of DM stacks. The proof of Theorem \ref{thm:main} will be given in In Section \ref{sec:proofs}.

\noindent {\bf Acknowledgement}: I thank I. Biswas, N. Fakhruddin and C. Xu for useful discussions. Part of this work was done during my visit to Harish-Chandra Research Institute (HRI). I thank HRI for the invitation and hospitality. 

\section{Preliminaries on the Brauer group of DM stacks}\label{section:prel}

\noindent In this section we recall/prove some basic facts (see (\ref{dmbasic}), (\ref{thm:purity}) and (\ref{thm:brauersheaf})) about the Brauer group of smooth DM stacks.\\

\noindent The following theorem is well known. The part $(i)$ of the theorem below is due to  ~M.~Lieblich (\cite{twistedlieblich},$3.1.3.3$), and $(iii)$ follows easily from results of ~A.~Vistoli and ~A.~Kresch \cite{vistoli-kresch}. 

\begin{thm}\label{dmbasic}Let $X/k$ be a smooth separated integral DM stack. Let $\eta$ be the residual gerbe at its generic point. Then
\begin{enumerate}
\item[(i)]$H^2(X,\G_m) \to H^2(\eta,\G_m)$ is injective. 
\item[(ii)] $H^2(X,\G_m)$ is a torsion group. 
\item[(iii)] If $X/k$ is quasiprojective, then $Br(X)'= H^2(X,\G_m)'$.
\end{enumerate}
\end{thm}

\begin{lem}\label{lem:gerbetorsion}
Let $G$ be a finite group and $X$ be a $G$-gerbe over a field $k$. Then for any sheaf of abelain groups $\sF$ on $X_{et}$, $H^i(X,G)$ is a torsion group for all $i>0$. 
\end{lem}
\begin{proof}
In the special case when $X=\Spec(k)/G$ is a neutral gerbe, the sheaf cohomology groups can be identified with $H^i(G\times\Gal(k),-)$ which are torsion for all $i>0$. In the general case, let $L/k$ be a galois extension, with $\Gal(L/k)=\Gamma$, such that $X_L=X\times_kL$ is a neutral gerbe. Let $f:X_L\to X$ be the natural map, which is in fact a $\Gamma$-principal bundle.  Since $f_*$ is exact and takes injective sheaves to injective sheaves (by exactness of $f^{-1}$), for any sheaf $\sG$ on $X_L$, $H^i(X_L,\sG)\cong H^i(X,f_*\sG) \ \forall \ i$. The result now follows from the existence of a map $g_*g^*\sF \to \sF$ such that the composite $\sF \to g_*g^*\sF \to \sF$ is multiplication by order of $G$.
\end{proof}

\begin{proof}[Proof of Theorem \ref{dmbasic}] $(i)$ We repeat the argument from (\cite{twistedlieblich},$3.1.3.3$). Let $\alpha$ belong to the kernel of the map $( H^2(X,\G_m)\to H^2(\eta,\G_m))$ and $g:Y\to X$ be the $\G_m$-gerbe associated to $\alpha$. 
Since $\alpha$ goes to zero in $H^2(\eta,\G_m)$, there exists a twisted line bundle on the generic fiber of $g$, and hence on an open subset of $Y$. A coherent reflexive extension of this twisted line bundle is again a twisted line bundle since $Y$ is regular. Thus $\alpha$ must be zero. \\
\noindent $(ii)$ follows from \ref{lem:gerbetorsion} and $(i)$.  \\
\noindent $(iii)$ Let $\alpha \in H^2(X,\G_m)'$ and $g:Y\to X$ be the $\G_m$ gerbe associated to $\alpha$. We may assume the order of $\alpha$ is $p$, a prime number invertible in $k$. Lift $\alpha$ to a class $\tilde{\alpha} \in H^2(X,\mu_p)$. Let $\tilde{g}:\tilde{Y} \to X$ be the $\mu_p$-gerbe associated to $\alpha$. By (\cite{vistoli-kresch},$2.2$) and Gabber's theorem, $\tilde{Y}$ is a quotient stack and hence there exists a locally free twisted sheaf $\sE$ on $\tilde{Y}$. Then $\sE nd_{\sO_{\tilde{Y}}}(\sE)$ is pull back of an Azumaya algebra on $X$ having class $\alpha$ thus showing $\alpha$ is in the image of $Br(X)'$. 
\end{proof}

\noindent Recall the following purity theorem from etale cohomology. 

\begin{thm}[Purity]{\rm (\cite{milne},VI.$5$)}\label{thm:purityschemes}
Let $k$ be a field and $X/k$ be a smooth variety. Let $U\subset X$ be an open subset whose complement has codimension at least $2$. Let $n$ be any integer invertible in $k$ and $\sF$ be a locally constant $n$-torsion sheaf. Then
$$ H^i(X,\sF)\to H^i(U,\sF)$$
is an isomorphism for $i=1,2$.
\end{thm}

\noindent As an easy consequence of Leray spectral sequence, one can show that purity is a property which is etale local on $X$. Thus once the above result is true for varieties, it is automatically true for DM stacks. 

\begin{cor}
 Let $X/k$ be a smooth separated DM stack and $j:U\inj X$ be an open substack such that $X\backslash U$ has codimension at least $2$. Then for any integer $n$ invertible in $k$ and any locally constant $n$-torsion sheaf $\sF$, we have 
$$H^i(X,\sF)\to H^i(U,\sF)$$
is an isomorphism for $i=1,2$. 
\end{cor}
\begin{proof}
By the above theorem $R^ij_*\sF=0$ for $i=1,2$. Thus the result follows immediately from the Leray spectral sequence for $j_*\sF$. 
\end{proof}

\noindent Our main interest in purity is the following consequence for the Brauer group of DM stacks. For any an abelian group $A$, we denote its $n$-torsion subgroup by $A[n]$.

\begin{cor}\label{thm:purity}
Let $X/k$ be a smooth DM stack and $j:U\inj X$ be an open subset whose complement has codimension at least two. Then $H^2(X,\G_m)'\to H^2(U,\G_m)'$ is an isomorphism. 
\end{cor}
\begin{proof}
For an integer $n$, invertible in $k$, we need to show $H^2(X,\G_m)[n]\to H^2(U,\G_m)[n]$ is an isomorphism. We have the following commutative diagram with exact rows.

$$\xymatrix{
Pic(X)\ar[r]^n\ar[d] & Pic(X)\ar[r]\ar[d] & H^2(X,\mu_n) \ar[r]\ar[d] &  H^2(X,\G_m)[n] \ar[r]\ar[d] & 0 \\
Pic(U)\ar[r]^n       & Pic(U)\ar[r]       & H^2(U,\mu_n) \ar[r]       &  H^2(U,\G_m)[n] \ar[r]       & 0
}$$
Since $X\backslash U$ has codimension at least $2$, $Pic(X)\to Pic(U)$ is an isomorphism. Moreover, $H^2(X,\mu_n)\to H^2(U,\mu_n)$ is also an isomorphism by the purity theorem mentioned above. Thus $H^2(X,\G_m)[n]\to H^2(U,\G_m)[n]$ is an isomorphism.
\end{proof}

\begin{thm}\label{thm:brauersheaf}
Let $X/k$ be any smooth orbifold. Then $U\to H^2(U,\G_m)'$ is a Zariski sheaf on $X$. 
\end{thm}

\begin{lem}\label{lem:multiple}
Let $X$ be a smooth orbifold and $h:X\to \underline{X}$ be its coarse moduli space. Let $D$ be an irreducible divisor on $\underline{X}$ and $p$ the generic point of $D$. Then $h^*(D) = s_p \tilde{D}$ for a prime Weil divisor $\tilde{D}$ on $X$.
\end{lem}
\begin{proof}
The question is etale local in nature at $p$. Thus we may base change to the henselisation of the local ring at $p$ and assume that $\underline{X}=\Spec(A)$, $X=\Spec(B)/\Gamma$ where $(A,m)$ and $(B,\eta)$ are discrete valuation rings with quotient fields $L$ and $K$ respectively, and $\Gamma$ is a finite group acting on $B$. Moreover, $A=B^G$. Note that since $X$ is an orbifold, action of $\Gamma$ on $L$ is faithful. Thus $\Gamma = \Gal(L/K)$. The cardinality of the inertia subgroup of $\Gamma$ is exactly equal to $s_p$. Moreover since this is the ramification index of $B/A$, the lemma follows since $mB = \eta^{s_p}$.
\end{proof}

\begin{lem}\label{lem:vanishing}
Let $X/k$ be a smooth orbifold. Let $q:X\to \underline{X}$ be its coarse moduli space and $\check{q}$ denote the continuous functor from the etale site of $X$ to the Zariski site of $\underline{X}$. Then 
$$ H^i(\underline{X}_{Zar},R^1\check{q}_*\G_m)=0 \ \ \forall \ \ i>0$$
\end{lem}
\begin{proof}{\it Step 1}: Let $p$ be a codimension one point of $X$. We denote the residue field at $p$ by $k(p)$ and (by abuse of notation) denote the natural map $\Spec(k(p))\stackrel{p}{\longrightarrow} \underline{X}$ by $p$ itself. Note that for every such $p$, the Zariski sheaf $p_*(\Z/s_p\Z)$ is acyclic, i.e. $H^i(\underline{X}_{Zar},p_*\Z/s_p\Z)=0$ for $i\geq 1$. Thus to prove the lemma, it is enough to prove the following isomorphism
$$R^1\check{q}_*\G_m \cong \oplus p_*(\Z/s_p\Z) \ \ (\text{where} \ p\in X^{(1)})$$

\noindent {\it Step 2}: We may assume $X$ is integral without loss of generality and we let $\eta:\Spec(K) \to X$ be the generic point of $X$. Let  $\G_{m,K}$ (resp. $\G_{m,X}$) be the sheaf defined by $\G_m$ on the etale sites of $\Spec(K)$ (resp. $X$).
We claim that $R^1\check{q}_*(\eta_*\G_{m,K})=0$. To show this, it is enough to show that for any open subset $U$ of $X$, $H^1(U,\eta_*\G_{m,K})=0$. 
But by Leray spectral sequence for $\eta_*\G_{m,K}$, we have an injection $H^1(U,\eta_*\G_{m,K})\inj H^1(\Spec(K),\G_{m,K})$, and hence the required vanishing is an easy consequence of Hilbert theorem $90$. \\

\noindent {\it Step 3}: We have the following short exact sequence of etale sheaves on $X$
$$ 0 \to \G_{m,X}\to \eta_*\G_{m,K} \to \underline{\Div}_X \to 0$$
where $\underline{\Div}_X$ denotes the sheaf on $X$ of Weil divisors. By the above step, this short exact sequence gives rise to the following exact sequence of Zariski sheaves on $\underline{X}$
$$ K^*\to \check{q}_*\underline{\Div}_X \to R^1\check{q}_*\G_{m,X} \to 0$$
where $K^*$ denotes the constant sheaf $K^*$ on $\underline{X}_{Zar}$. Since $\underline{X}$ is normal, cokernel of $K^*\to \check{q}_*\underline{\Div}_X$ is the naturally isomorphic to the cokernel of $\underline{\Div}_{\underline{X}}\to \check{q}_*\underline{\Div}_X$. We leave it to the reader to check using (\ref{lem:multiple}) that this cokernel is isomorphic to $\oplus p_*(\Z/s_p\Z)$.
\end{proof}

\begin{proof}[Proof of Theorem \ref{thm:brauersheaf}]
We first reduce to the case when $\underline{X}$ is also smooth over $k$. Note that since $X$ is normal, $\underline{X}$ is also normal. Thus if $Z$ is the singular locus of $\underline{X}$, then codimension of $Z$ in $\underline{X}$ is at least $2$. Let $\tilde{Z}=q^{-1}(Z)$. By purity, for any open substack $U$ of $X$, $H^2(U,\G_m)\to H^2(U\backslash (\tilde{Z}\intersection U),\G_m)$ is an isomorphism. Thus we may replace $X$ by $X\backslash \tilde{Z}$, $\underline{X}$ by $\underline{X}\backslash Z$ and assume $\underline{X}$ is smooth. \\

\noindent To prove the theorem, it is enough to show that for every Zariski open subset $U\subset \underline{X}$, 
$$H^2(q^{-1}(U),\G_m)\to H^0(U_{Zar},R^2\check{q}_*\G_m)$$
is an isomorphism. Without loss of generality, we replace $X$ by $q^{-1}(U)$.  The required isomorphism follows easily from Leray spectral sequence, by using Lemma (\ref{lem:vanishing}) and the fact that $H^i(\underline{X}_{Zar},\G_m)=0$ for $i\geq 2$ (since $\underline{X}$ is smooth).
\end{proof}

\section{Proof of the main theorem}\label{sec:proofs}
\noindent In this section we prove Theorem \ref{thm:main}. For any codimension one point $p$ of a DM stack $X$, we use the following notation. \\ 

\begin{tabular}{rl}
$i_p:\sG(p) \inj X$ 		& : residual gerbe at $p$ \\
$k(p)$ 				& : residue field at $p$ \\
$f_p:\Spec(k(p)) \to \underline{X}$ 	& : the natural morphism\\
$D_p$				& : the prime Weil divisor with generic point $p$
\end{tabular}

\begin{lem}\label{lem:easyway}
Let $X/k$ be a smooth quasiprojective orbifold. We have an exact sequence 
$$ 0 \to Br(X) \to Br(\kappa(X))\to  \oplus_{p\in X^{(1)}} H^2(\sG(p),\Z)$$
The above sequence is functorial in the following sense. If $F:Y\to X$ is any $1$-morphism of smooth orbifolds of same dimension, such that any the image of any codimension one point of $Y$ is again a codimension one point of $X$, then the following diagram commutes
$$\xymatrix{
0 \ar[r] & Br(X) \ar[r]\ar[d] & Br(\kappa(X))\ar[r]\ar[d] &  \oplus_{p\in X^{(1)}} H^2(\sG(p),\Z) \ar[d]^{F^*} \\
0 \ar[r] & Br(Y) \ar[r] &  Br(\kappa(Y)) \ar[r] & \oplus_{q\in Y^{(1)}}H^2(\sG(q),\Z)
}$$
where the last vertical arrow $F^*$ has the following properties :
\begin{enumerate}
\item[(i)] If $q$ does not lie over $p$, then the induced map $H^2(\sG(p),\Z)\to H^2(\sG(q),\Z)$ is zero.
\item[(ii)]If $F^*(D_p)=e\cdot D_q$ (as Weil divisors), then the induced map $H^2(\sG(p),\Z)\to H^2(\sG(q),\Z)$ factors through 
$H^2(\sG(q),\Z)\stackrel{e}{\longrightarrow} H^2(\sG(q),\Z)$.
\end{enumerate}
\end{lem}
\begin{proof}
The above statements are well known for schemes and the proofs for the stack are similar, hence we only indicate the main ingredients of the argument. First, because of purity (\ref{thm:purity}) and Theorem (\ref{thm:brauersheaf}) it is enough to prove both the above statements assuming $X$ is a DM stack whose coarse moduli space is $\Spec(A)$ where $(A,m)$ is a DVR. Let $\eta:\Spec(K) \to X$ denote the generic point. In this case we let $p$ denote the unique codimension one point of $X$. We have the following short exact sequence of etale sheaves on $X$.
$$ 0 \to \G_{m,X} \to \eta_*\G_{m,K} \to i_{p*}\Z \to 0$$
The existence of the claimed exact sequence now follows after using the following identifications
\begin{enumerate}
\item $H^1(X,i_{p*}\Z)=0$.
\item $H^2(X,i_{p*}\Z)=H^2(\sG(p),\Z)$.
\item $H^2(X,\eta_*\G_{m,K})=Br(K)$.
\end{enumerate}
By Theorem (\ref{thm:brauersheaf}), functoriality, and the properties $(i),(ii)$ satisfied by $F^*$ can also be checked 'Zariski locally' at codimension one points of $X$ and $Y$, in which case one may assume without loss of generality that coarse moduli spaces of both $X$ and $Y$ are spectrums of discrete valuation rings. Details are left to the reader.
\end{proof}

\begin{proof}[Proof of Theorem \ref{thm:main}]{\it Step 1} : The coarse moduli space $\underline{X}$ is normal and hence its singular locus is of codimension at least $2$. Thus by Theorem \ref{thm:purity}, $Br(h^{-1}(\underline{X}_{sm})) \to Br(X)$ is an isomorphism. Therefore without loss of generality, we replace $X$ by $h^{-1}(\underline{X}_{sm})$ and $\underline{X}$ by $\underline{X}_{sm}$ and assume that the coarse moduli space $\underline{X}$ is in fact smooth. \\

\noindent {\it Step 2} : Let  $\alpha \in Br(\kappa(X))$ such that $s_p\cdot \alpha$ is unramified at $p \ \forall \ p \in \underline{X}^{(1)}$. We will show that $\alpha \in Br(X)$. By (\ref{lem:easyway}) we have a commutative diagram
$$\xymatrix{
0\ar[r] & Br(\underline{X}) \ar[r]\ar[d] & Br(\kappa(X)) \ar[r]\ar@{=}[d] & \oplus H^2(\Spec(k(p)),\Z) \ar[d]\\
0\ar[r] & Br(X)\ar[r]			 & Br(\kappa(X)) \ar[r]		  & \oplus H^2(\sG(p),\Z)
}$$
To show $\alpha \in Br(X)$, it is enough to show that the image of $\alpha$ in $\oplus H^2(\Spec(k(p)),\Z)$ maps to zero in $\oplus H^2(\sG(p),\Z)$. This follows from (\ref{lem:multiple}) and the property $(ii)$ mentioned in (\ref{lem:easyway}).\\

\noindent {\it Step 3} : What remains to show is that if $\alpha \in Br(X)\subset Br(\kappa(\underline{X}))$ then  $s_p\alpha$ is unramified at $p$ for all $p\in \underline{X}^{(1)}$. This question may be proved Zariski locally at each $p$. By replacing $\underline{X}$ by the spectrum of the local ring at $p$, we may assume, $\underline{X}=\Spec(A)$ where $(A,m)$ is a discrete valuation ring. By \cite{vistoli-kresch} and (\cite{kresch},$5.2$), $X$ must be a quotient of a semilocal dedekind domain $B$ by a finite group $\Gamma$. Let $K$ and $L$ denote the quotient fields of $A$ and $B$ respecitively. Note that since $X$ is an orbifold, $\Gamma = \Gal(L/K)$. To show $s_p\alpha$ is unramified class in $Br(K)$, we may base extend to the henselisation of $A$. Thus we may assume both $A$ and $B$ are henselian discrete valuation rings. Further, by replacing $B$ by a suitable finite extension, we may assume that the pull back of the class $\alpha$ to $\Spec(B)$ is trivial.  \\

\noindent {\it Step 4} : Both, $Br(X)$ and $Br(A)$ are subgroups of $Br(K)$. Since the pull back of $\alpha$ to $B$, and hence also to $L$ is trivial, $\alpha$ actually lies in $H^2(\Gamma,L^*) \subset Br(K)$. We claim that $\alpha \in H^2(\Gamma,L^*)$ is actually contained in the image of the map $$ H^2(\Gamma, B^*) \to H^2(\Gamma,L^*)$$
The class $\alpha$ is represented by a $PGL(n)$ bundle on $X$ (for some $n$), which gives rise to a $\Gamma$-equivariant $PGL_n$ bundle on $\Spec(B)$. Since the isomorphism classes of $\Gamma$-equivariant $PGL_n$ bundles, whose underlying $PGL(n)$ bundle is trivial, are in bijection with the cohomology set $H^1(\Gamma,PGL_n(B))$, $\alpha$ defines a class in $H^1(\Gamma,PGL_n(B))$. The claim follows easily from the following commutative diagram 
$$ \xymatrix{
H^1(\Gamma,PGL_n(B)) \ar[r]\ar[d] & H^1(\Gamma,PGL_n(L))\ar[d] \\
H^2(\Gamma,B^*) \ar[r] & H^2(\Gamma,L^*)
}$$
Note that the existence of the first vertical map is because there are no nontrivial line bundles on $B$ and hence the following sequence of $\Gamma$-modules is exact
$$ 1 \to B^* \to GL_n(B) \to PGL_n(B) \to 1$$

\noindent {\it Step 5} : Let $L_{un}$ (resp. $K_{un}$) denote the maximal unramified extensions of $L$ and $K$ respectively. Let $\tilde{G}=\Gal(L_{un}/K)$ and $G=\Gal(K_{un}/K)$. Let $v_L$ and $v_K$ denote the surjective valuations of $L_{un}$ and $K_{un}$ respectively. Since the ramification index of $L/K$ is exactly equal to $s_p$, the following diagram 
commutes 
$$\xymatrix{
K_{un}^* \ar[d]\ar[r]^{v_K} & \Z \ar[d]^{s_p}\\
L_{un}^* \ar[r]^{v_L} & \Z 
}$$
which gives rise to 
$$\xymatrix{
H^2(G,K_{un}^*) \ar[r]^{v_K}\ar[d] & H^2(G,\Z) \ar[d]^{\theta} \\ 
H^2(\tilde{G},L_{un}^*) \ar[r]^{v_L} & H^2(\tilde{G},\Z)  
}$$
where the map $\theta$ is a composite of the inflation map $$H^2(G,\Z) \stackrel{inf}{\longrightarrow} H^2(\tilde{G},\Z)$$ followed by the multiplication map
$$ H^2(\tilde{G},\Z) \stackrel{s_p}{\longrightarrow} H^2(\tilde{G},\Z)$$ 
But if $H$ denotes the kernel of $\tilde{G}\to G$, then the vanishing of $H^1(H,\Z)$ implies that $H^2(G,\Z) \stackrel{inf}{\longrightarrow} H^2(\tilde{G},\Z)$ is injective (see \cite{serre},$2.6$). Thus the kernel of $\theta$ is precisely the subgroup of $s_p$-torsion classes in $H^2(G,\Z)$. 
 
\noindent The kernel of $H^2(G,K_{un}^*) \stackrel{v_K}{\longrightarrow} H^2(G,\Z)$ is exactly $Br(A)$. Thus, showing $s_p\alpha$ is unramified, is equivalent to showing that the image of $\alpha$ in $H^2(\tilde{G},\Z)$ is zero. Using the exact sequence of $\Gamma$ modules 
$$ 0 \longrightarrow B^* \longrightarrow L^* \stackrel{v_L}{\longrightarrow} \Z \to  0 $$
we see that any class in $H^2(\Gamma,L^*)$ coming from $H^2(\Gamma,B^*)$ maps to zero in $H^2(\Gamma,\Z)$. The proof now follows from the following commutative diagram.
$$\xymatrix{
H^2(\Gamma,L^*) \ar[r]^{v_L}\ar[d]_{inf} & H^2(\Gamma,\Z) \ar[d]^{inf} \\
H^2(\tilde{G},L_{un}^*) \ar[r]^{v_L} & H^2(\tilde{G},\Z) 
}
$$
\end{proof}



\vspace{1cm}
\noindent Amit Hogadi\\
\noindent Tata Institute of Fundamental Research,\\
\noindent Homi Bhabha Road, Colaba, Mumbai 400005. India.\\
\noindent {\it Email}: amit@math.tifr.res.in
\end{document}